\documentclass[12pt,a4paper]{amsart}
\usepackage{mathrsfs}

\newtheorem{theo+}              {Theorem}           [section]
\newtheorem{prop+}  [theo+]     {Proposition}
\newtheorem{coro+}  [theo+]     {Corollary}
\newtheorem{lemm+}  [theo+]     {Lemma}
\newtheorem{exam+}  [theo+]     {Example}
\newtheorem{rema+}  [theo+]     {Remark}
\newtheorem{defi+}  [theo+]     {Definition}

\newenvironment{theorem}{\begin{theo+}}{\end{theo+}}
\newenvironment{proposition}{\begin{prop+}}{\end{prop+}}
\newenvironment{corollary}{\begin{coro+}}{\end{coro+}}

\usepackage{amsthm}
\theoremstyle{plain} \theoremstyle{remark}
\newtheorem{remark}{Remark}
\newtheorem{example}{Example}

\def \r{\mbox{${\mathbb R}$}}
\def\E{/\kern-1.0em \equiv }

\title{ Biharmonic hypersurfaces in Riemannian manifolds}
\author{Ye-Lin Ou$^{*}$ }

\address{Department of
Mathematics,\newline\indent Texas A $\&$ M University-Commerce,
\newline\indent Commerce TX 75429,\newline\indent USA.\newline\indent
E-mail:yelin$\_$ou@tamu-commerce.edu}
\thanks{$^*$ Supported by Texas A $\&$ M University-Commerce
``Faculty Development Program" (2008)}

\begin{document}
\title[Biharmonic hypersurfaces in Riemannian manifolds]{Biharmonic hypersurfaces in  Riemannian manifolds }
\date {10/28/2009} 
 \maketitle

\section*{Abstract}

\begin{quote}
{\footnotesize We study biharmonic hypersurfaces in a generic
Riemannian manifold. We first derive an invariant equation for such
hypersurfaces generalizing the biharmonic hypersurface equation in
space forms studied in \cite{Ji2}, \cite{CH}, \cite{CMO1},
\cite{CMO2}. We then apply the equation to show that the generalized
 Chen's conjecture is true for totally umbilical biharmonic
hypersurfaces in an Einstein space, and construct a (2-parameter)
family of conformally flat metrics and a (4-parameter) family of
multiply warped product metrics each of which turns the foliation of
an upper-half space of $\mathbb{R}^m$ by parallel hyperplanes into a
foliation with each leave a proper biharmonic hypersurface. We also
study the biharmonicity of Hopf cylinders of a Riemannian
submersion.}
\end{quote}
\section{Biharmonic maps and submanifolds}

All manifolds, maps, and tensor fields that appear in this paper are supposed to be smooth unless there is an otherwise statement.\\

A biharmonic map is a map $\varphi:(M, g)\longrightarrow (N, h)$
between Riemannian manifolds that is a critical point of the
bienergy functional
\begin{equation}\nonumber
E^{2}\left(\varphi,\Omega \right)= \frac{1}{2} {\int}_{\Omega}
\left|\tau(\varphi) \right|^{2}{\rm d}x
\end{equation}
for every compact subset $\Omega$ of $M$, where $\tau(\varphi)={\rm
Trace}_{g}\nabla {\rm d} \varphi$ is the tension field of $\varphi$.
The Euler-Lagrange equation of this functional gives the biharmonic
map equation (\cite{Ji})
\begin{equation}\label{BTF}
\tau^{2}(\varphi):={\rm
Trace}_{g}(\nabla^{\varphi}\nabla^{\varphi}-\nabla^{\varphi}_{\nabla^{M}})\tau(\varphi)
- {\rm Trace}_{g} R^{N}({\rm d}\varphi, \tau(\varphi)){\rm d}\varphi
=0,
\end{equation}
which states the fact that the map $\varphi$ is biharmonic if and
only if its bitension field $\tau^{2}(\varphi)$ vanishes
identically. In the above equation we have used $R^{N}$ to denote
the curvature operator of $(N, h)$ defined by
$$R^{N}(X,Y)Z=
[\nabla^{N}_{X},\nabla^{N}_{Y}]Z-\nabla^{N}_{[X,Y]}Z.$$ Clearly, it
follows from (\ref{BTF}) that any harmonic map is biharmonic and we
call those non-harmonic biharmonic maps {\bf proper biharmonic maps}.\\

\indent For a submanifold $M^m$ of Euclidean space $\mathbb{R}^{n}$
with the mean curvature vector ${\bf H}$ viewed as a map ${\bf H} :
M\longrightarrow \mathbb{R}^{n}$, B. Y. Chen \cite{CH} called it a
biharmonic submanifold if $\triangle {\bf H}=(\Delta H^{1},\ldots,
\Delta H^{n})=0$, where $\Delta$ is the Beltrami-Laplace operator of
the induced metric on $M^m$. Note that if we use ${\bf i}:
M\longrightarrow \mathbb{R}^{n}$ to denote the inclusion map of the
submanifold, then the tension field of the inclusion map ${\bf i}$
is given by $\tau ({\bf i})=\Delta {\bf i} = m{\bf H}$ and hence the
submanifold $M^n\subset \r^n$ is biharmonic if and only if
$\triangle {\bf H} =\triangle(\frac{1}{m}\triangle {\bf
i})=\frac{1}{m}\triangle^{2}{\bf i}=\frac{1}{m}\tau^{2}({\bf i})=0$,
i.e., the inclusion map is a biharmonic map. In general, a
submanifold $M$ of $(N,h)$ is called a {\bf biharmonic submanifold}
if the inclusion map ${\bf i}: (M, {\bf i}^*h)\longrightarrow (N,h)$
is biharmonic isometric immersion. It is well-known that an
isometric immersion is minimal if and only if it is harmonic. So a
minimal submanifold is trivially biharmonic and we call a
non-minimal
biharmonic submanifold a {\bf proper biharmonic submanifold}.\\

Here are some known facts about biharmonic submanifolds:
\begin{itemize}
\item[1.]{\em  Biharmonic submanifolds in Euclidean spaces}: Jiang
\cite{Ji2}, Chen-Ishikawa \cite{CI} proved that any biharmonic
submanifold in $\r^{3}$ is minimal; Dimitri$\acute{\rm c}$ \cite{Di}
showed that any biharmonic curves in $\r^n$ is a part of a straight
line, any biharmonic submanifold of finite type in $\r^n$ is
minimal, any pseudo-umbilical submanifolds $M^m\subset \r^n$ with
$m\ne 4$ is minimal, and any biharmonic hypersurface in $\r^n$ with
at most two distinct principal curvatures is minimal; it is proved
in \cite{HV} that any biharmonic hypersurface in $\r^4$ is minimal.
Based on these, B. Y. Chen \cite{CH} proposed the conjecture: any
biharmonic submanifold of Euclidean space is minimal, which is still
open.
\item[2.]{\em Biharmonic submanifolds in hyperbolic space forms}:
Caddeo, Montaldo and Oniciuc \cite{CMO2} showed that any biharmonic
submanifold in hyperbolic $3$-space $H^{3}(-1)$ is minimal, and
pseudo-umbilical biharmonic submanifold $M^m\subset H^n$ with $m\ne
4$ is minimal. It is shown in \cite{BMO} that any biharmonic
hypersurface of $H^n$ with at most two distinct principal curvatures
is minimal. Based on these, Caddeo, Montaldo and Oniciuc \cite{CMO1}
extended Chen's conjecture to be {\bf the generalized Chen's
conjecture}: any biharmonic submanifold  in $(N, h)$ with ${\rm
Riem}^N\leq 0$ is minimal. \item[3.]{\em Biharmonic submanifolds in
spheres}: The first example of proper biharmonic submanifold in
$S^{n+1}$ was found (\cite{Ji3}) to be the generalized Clifford
torus $S^p(\frac{1}{\sqrt{2}})\times S^{q}(\frac{1}{\sqrt{2}})$ with
$p\ne q, p+q=n$. The second type of the proper biharmonic
submanifolds in $S^{n+1}$ was found in \cite{CMO1} to be hypersphere
$S^{n}(\frac{1}{\sqrt{2}})$. The authors in \cite{CMO1} also gave a
complete classification of biharmonic submanifolds in $S^3$. It was
proved in \cite{BMO} that any pseudo-umbilical biharmonic
submanifold $M^m\subset S^{n+1}$ with $m\ne 4$ has constant mean
curvature whilst in \cite{BMO} the same authors showed that a
hypersurface $M^{n}\subset S^{n+1}$ with at most two distinct
principal curvatures (which, for $n>3$, is equivalent to saying that
$M$ is a quasi-umbilical or conformally flat hypersurface in
$S^{n+1}$ \cite{NM}) is biharmonic, then $M$ is an open part of  the
hypersphere $S^{n}(\frac{1}{\sqrt{2}})$, or the generalized Clifford
torus $S^p(\frac{1}{\sqrt{2}})\times S^{q}(\frac{1}{\sqrt{2}})$ with
$p\ne q, p+q=n$. Some example of proper biharmonic real
hypersurfaces in $CP^n$ were found and all proper biharmonic tori
$T^{n+1}=S^{1}(r_1)\times S^{1}(r_2)\times\ldots \times
S^{1}(r_{n+1})$ in $S^{2n+1}$ were determined in \cite{Zh}. All the
known examples of biharmonic submanifolds in spheres lead to
 {\bf the conjecture} \cite{BMO}: any biharmonic
submanifold in sphere has constant mean curvature; and any proper
biharmonic hypersurface in $S^{n+1}$ is an open part of  the
hypersphere $S^{n}(\frac{1}{\sqrt{2}})$, or the generalized Clifford
torus $S^p(\frac{1}{\sqrt{2}})\times S^{q}(\frac{1}{\sqrt{2}})$ with
$p\ne q, p+q=n$.  \item[4.]{\em Biharmonic submanifolds in other
model spaces}: For the study of biharmonic curves in various model
spaces we refer the readers to the survey article \cite{MO}, and for
special biharmonic submanifolds in contact manifolds or Sasakian
space forms see recent works \cite {AEMS}, \cite{In}, \cite{FO1},
\cite{FO2}, \cite{Sa1}, and \cite{Sa2}. Some constructions and
classifications of biharmonic surfaces in three-dimensional
geometries will appear in \cite{OW}.  \item[5.]{\em Biharmonic
submanifolds in other senses}: We would like to point out that some
authors (as in \cite{VMB}) use $\Delta {\bf H}=0$ to define a
``biharmonic submanifold" in a Riemannian manifold, which agree with
our notion of biharmonic submanifold only if the ambient space is
 flat. For conformal biharmonic submanifolds (i.e., conformal
biharmonic immersions) see \cite{Ou}.
\end{itemize}

In this paper, we study biharmonic hypersurfaces in a generic
Riemannian manifold. In Section 2, we derive an invariant equation
for biharmonic hypersurfaces in a Riemannian manifold that involves
the mean curvature function, the norm of the second fundamental
form, the shape operator of the hypersurface, and the Ricci
curvature of the ambient space, and prove that the generalized
Chen's conjecture is true for totally umbilical hypersurfaces in an
Einstein space. Section 3 is devoted to construct a family of
conformally flat metrics and a family of multiply warped product
metrics each of which turns the foliation of an upper-half space of
$\r^m$ by parallel hyperplanes into a foliation with each leave a
proper biharmonic hypersurface. These are accomplished by starting
with hyperplanes in Euclidean space then looking for a special type
of conformally flat or multiply warped product metrics on the
ambient space that reduce the biharmonic hypersurface equation into
ordinary differential equations whose solutions give the metrics
that render the inclusion maps proper biharmonic isometric
immersions. Finally, we study biharmonicity of Hopf cylinders given
by a Riemannian submersion from a complete $3$-manifold in Section
4. Our method shows that there is no proper biharmonic Hopf cylinder
in $S^3$ which recovers Proposition 3.1 in \cite{In}.

\section{The equations of biharmonic hypersurfaces}

Recall that if $\varphi : M\longrightarrow (N, h)$ is the inclusion
map of a submanifold, or more generally, an isometric immersion,
then we have an orthogonal decomposition of the vector bundle
$\varphi^{-1}TN=\tau M\oplus \nu M $ into the tangent and normal
bundles. We use ${\rm d}\varphi$ to identify $TM$ with its image
$\tau M$ in $\varphi^{-1}TN$. Then, for any $X,Y\in \Gamma(TM)$ we
have $\nabla^{\varphi}_{X}({\rm d}\varphi(Y))=\nabla^{N}_{X}Y$,
whereas ${\rm d}\varphi(\nabla^{M}_{X}Y)$ equals the tangential
component of $ \nabla^{N}_{X}Y$. It follows that
\begin{eqnarray}\label{T2F}
\nabla {\rm d}\varphi(X,Y)=\nabla^{\varphi}_{X}({\rm
d}\varphi(Y))-{\rm d}\varphi(\nabla^{M}_{X}Y)=B(X,Y),
\end{eqnarray}
i.e., the second fundamental form $\nabla {\rm d}\varphi(X,Y)$ of
the isometric immersion  $\varphi$  agrees with the second
fundamental form $B(X,Y)$ of the immersed submanifold $\varphi(M)$
in $N$ (see \cite{KN}, Chapter 7, also \cite{BW1}, Example 3.2.3 for
details). From (\ref{T2F}) we see that the tension field
$\tau(\varphi)$ of an isometric immersion and the mean curvature
vector field $\eta$ of the submanifold are related by
\begin{equation}
\tau(\varphi)=m\eta.
\end{equation}

For a hypersurface, i.e, a codimensional one isometric immersion
$\varphi:M^{m}\longrightarrow N^{m+1}$, we can choose a local unit
normal vector field $\xi$ to $\varphi(M) \subset N$. Then,
$\eta=H\xi$ with H being the mean curvature function, and we can
write $B(X,Y)=b(X,Y)\xi$, where $b:TM\times TM\longrightarrow
C^{\infty}(M)$ is the function-valued second fundamental form. The
relationship between the shape operator $A$ of the hypersurface with
respect to the unit normal vector field $\xi$ and the second
fundamental form is given by
\begin{eqnarray}\label{2FF}
B(X,Y)=\langle \nabla^N_XY, \xi\rangle\xi=-\langle Y,
\nabla^N_X\xi\rangle\xi=\langle AX, Y\rangle\xi,\\\label{2FF2}
\langle AX, Y\rangle=\langle B(X,Y), \xi\rangle =\langle b(X,Y)\xi,
\xi\rangle =b(X,Y).
\end{eqnarray}

\begin{theorem}\label{MTH}
Let $\varphi:M^{m}\longrightarrow N^{m+1}$ be an isometric immersion
of codimension-one with mean curvature vector $\eta=H\xi$. Then
$\varphi$ is biharmonic if and only if:
\begin{equation}\label{BHEq}
\begin{cases}
\Delta H-H |A|^{2}+H{\rm
Ric}^N(\xi,\xi)=0,\\
 2A\,({\rm grad}\,H) +\frac{m}{2} {\rm grad}\, H^2
-2\, H \,({\rm Ric}^N\,(\xi))^{\top}=0,
\end{cases}
\end{equation}
where ${\rm Ric}^N : T_qN\longrightarrow T_qN$ denotes the Ricci
operator of the ambient space defined by $\langle {\rm Ric}^N\, (Z),
W\rangle={\rm Ric}^N (Z, W)$ and  $A$ is the shape operator of the
hypersurface with respect to the unit normal vector $\xi$.
\end{theorem}
\begin{proof}
Choose a local orthonormal frame $\{{e_{i}}\}_{i=1,\ldots,m}$ on $M$
so that $\{{\rm d}\varphi(e_{1}), \ldots, {\rm d}\varphi(e_{m}),
\xi\}$ is an adapted orthonormal frame of the ambient space defined
on the hypersurface. Identifying  ${\rm
d}\varphi(X)=X,\;\;\nabla^{\varphi}_{X} W =\nabla^{N}_{X}W$ and
noting that the tension field of $\varphi$ is $\tau(\varphi)=m H
\xi$ we can compute the bitension field of $\varphi$ as:
\begin{equation}\label{J126}
\begin{array}{lll}
\tau^{2}(\varphi)=\sum\limits_{i=1}^{m}\{\nabla^{\varphi}_{e_{i}}\nabla^{\varphi}_{e_{i}}(mH\xi)
-\nabla^{\varphi}_{\nabla_{e_{i}}e_{i}}(mH\xi)-R^{N}({\rm
d}\varphi(e_{i}),mH\xi) d
\varphi(e_{i})\}\\
=m \sum\limits_{i=1}^{m}\{{e_{i}}e_{i}
(H)\xi+2{e_{i}}(H)\nabla^{N}_{e_{i}}\xi +
H\nabla^{N}_{e_{i}}\nabla^{N}_{e_{i}}\xi
-(\nabla_{e_{i}}e_{i})(H)\xi-H\nabla^{N}_{\nabla_{e_{i}}e_{i}}
\xi\}\\-mH\sum\limits_{i=1}^{m}R^{N}({\rm
d}\varphi(e_{i}), \xi) d \varphi(e_{i})\\
=m(\Delta H)\xi-2mA({\rm grad}\,H)
-mH\Delta^{\varphi}\xi-mH\sum\limits_{i=1}^{m}R^{N}({\rm
d}\varphi(e_{i}), \xi) d \varphi(e_{i}).
\end{array}
\end{equation}
To find the tangential and normal parts of the bitension field we
first compute the tangential and  normal components of the curvature
term to have
\begin{equation}
 \sum\limits_{i,k=1}^{m} \langle R^{N}({\rm d}\varphi(e_{i}),
\xi) d \varphi(e_{i}),e_k \rangle e_k=-[{\rm
Ric}^N(\xi,e_k)]e_k=-(Ric\,(\xi))^{\top},
\end{equation}
and
\begin{equation} \sum\limits_{i=1}^{m} \langle R^{N}({\rm
d}\varphi(e_{i}), \tau(\varphi)) d \varphi(e_{i}),\xi
\rangle=-mH{\rm Ric}^N(\xi,\xi).
\end{equation}

To find the normal part of $\Delta^{\varphi}\xi$ we compute:
\begin{equation}\label{LNor}
\langle\Delta^{\varphi}\xi,\xi
\rangle=\sum\limits_{i=1}^{m}\langle-\nabla^{N}_{e_{i}}\nabla^{N}_{e_{i}}
\xi+\nabla^{N}_{\nabla_{e_{i}}e_{i}} \xi,\xi
\rangle=\sum\limits_{i=1}^{m}\langle \nabla^{N}_{e_{i}}\xi,
\nabla^{N}_{e_{i}}\xi\rangle.
\end{equation}
On the other hand, using (\ref{2FF}) and (\ref{2FF2}) we have
\begin{eqnarray}\notag
|A|^2&=&\sum_{i,j=1}^m\langle
Ae_i,e_j\rangle^2=\sum\limits_{i,j=1}^{m}{\langle\nabla^{N}_{e_{i}}\xi,e_{j}
\rangle}^{2}=\sum\limits_{i=1}^{m}\langle\nabla^{N}_{e_{i}}\xi,\sum\limits_{j=1}^{m}\langle\nabla^{N}_{e_{i}}\xi,e_{j}
\rangle e_{j} \rangle\\\notag &=&
\sum_{i=1}^m\langle\nabla^{N}_{e_{i}}\xi, \nabla^{N}_{e_{i}}\xi
\rangle,
\end{eqnarray}
which, together with (\ref{LNor}),  implies that
\begin{equation}
(\Delta^{\varphi}\xi)^{\bot}=\langle\Delta^{\varphi}\xi,\xi
\rangle\xi=\sum\limits_{i=1}^{m}{\langle\nabla^{N}_{e_{i}}\xi,
\nabla^{N}_{e_{i}}\xi \rangle}\xi= |A|^{2}\xi.
\end{equation}
A straightforward computation gives the tangential part of
$\Delta^{\varphi}\xi$ as
\begin{eqnarray}\label{GD26}
(\Delta^{\varphi}\xi)^{\top} &=&
\sum\limits_{i,k=1}^{m}\langle-\nabla^{N}_{e_{i}}\nabla^{N}_{e_{i}}
\xi+\nabla^{N}_{\nabla_{e_{i}}e_{i}} \xi,e_k \rangle e_k\\\notag &=&
\sum\limits_{i,k=1}^{m}\langle
\nabla^{N}_{e_{i}}Ae_i-A(\nabla_{e_i}e_i),\; e_{k}\rangle e_k=
\sum\limits_{i,k=1}^{m}\Big[ (\nabla_{e_i}b)(e_k,\, e_i) \Big] e_k.
\end{eqnarray}
Substituting Codazzi-Mainardi equation for a hypersurface:
\begin{eqnarray}
&&(\nabla_{e_i}b)(e_k,e_i)-(\nabla_{e_k}b)(e_i,e_i)=(R^N(e_i,e_k)e_i)^{\bot}\\\notag
&=& \langle R^N(e_i,e_k)e_i, \xi\rangle
\end{eqnarray}
 into (\ref{GD26}) and using the normal coordinates at a point we have
\begin{eqnarray}
&&(\Delta^{\varphi}\xi)^{\top}=\sum\limits_{i,k=1}^{m}\Big[
(\nabla_{e_i}b)(e_k\, e_i) \Big] e_k\\\notag &=&
\sum_{k=1}^{m}\Big[\sum_{i=1}^{m}(\nabla_{e_k}b)(e_i,e_i) -{\rm
Ric}(\xi, e_k)\Big] e_k=m\, {\rm grad} \,(H )-[{\rm Ric}(\xi, e_k)]
e_k.
\end{eqnarray}
 Therefore, by collecting all the tangent and normal parts of the bitension field separately, we have
\begin{equation}\label{BOT}
(\tau^{2}(\varphi))^{\bot}=\langle\tau^{2}(\varphi),\xi
\rangle\xi=m\left(\Delta H-H|A|^{2}+H {\rm
Ric}^N(\xi,\xi)\right)\xi,
\end{equation}
and
\begin{eqnarray}\label{TOP}\notag
(\tau^{2}(\varphi))^{\top}&=&\sum\limits_{k=1}^{m}\langle\tau^{2}(\varphi),e_k
\rangle e_k\\ &=& -m\left( 2A\,({\rm grad}\,H) +\frac{m}{2} ({\rm
grad}\, H^2) -2\, H\,({\rm Ric}(\xi))^{\top}\right),
\end{eqnarray}
from which the theorem follows.
\end{proof}
As an immediate consequence of Theorem \ref{MTH} we have
\begin{corollary}
A constant mean curvature hypersurface in a Riemannian manifold is
biharmonic if and only if it is minimal or, ${\rm
Ric}^N(\xi,\xi)=|A|^2$ and $({\rm Ric}^N(\xi))^{\top}=0$. In
particular, we recovered Proposition 2.4 in \cite{Oni}  which states
that a constant mean curvature hypersurface in a Riemannian manifold
$(N^{m+1}, h)$ with nonpositive Ricci curvature is biharmonic if and
only if it is minimal.
\end{corollary}
\begin{corollary}
A hypersurface in an Einstein space $(N^{m+1},h)$ is biharmonic if
and only if its mean curvature function $H$ is a solution of the
following PDEs
\begin{equation}\label{Ein}
\begin{cases}
\Delta H-H\,|A|^{2}+ \frac{rH}{m+1}=0,\\
 2A\,({\rm grad}\,H) +\frac{m}{2}\, {\rm grad}\, H^2
=0,
\end{cases}
\end{equation}
where $r$ is the scalar curvature of the ambient space. In
particular, a hypersurface $\varphi :(M^{m}, g)\longrightarrow
(N^{m+1}(C), h)$ in a space of constant sectional curvature $C$ is
biharmonic if and only if its mean curvature function $H$ is a
solution of the following PDEs which was obtained by different
authors in several steps (see \cite{Ji2}, \cite{CH} and \cite{CMO2})
\begin{equation}\label{BHCon}
\begin{cases}
\Delta H-H\,|A|^{2}+ mCH=0,\\
 2A\,({\rm grad}\,H) +\frac{m}{2}\, {\rm grad}\, H^2=0.
\end{cases}
\end{equation}
\end{corollary}
\begin{proof}
It is well known that if $(N^{m+1},h)$ is an Einstein manifold then
${\rm Ric}^N (Z,W)=\frac{r}{m+1}h(Z,W)$ for any $Z, W \in TN$ and
hence $({\rm Ric}^N\,(\xi))^{\top}=0$ and ${\rm
Ric}^N(\xi,\xi)=\frac{r}{m+1}$. From these and Equation (\ref{BHEq})
we obtain Equation (\ref{Ein}). When $(N^{m+1}(C),h)$ is a space of
constant sectional curvature $C$, then it is an Einstein space with
the scalar curvature $r=m(m+1)C$. Substituting this into (\ref{Ein})
we obtain (\ref{BHCon}).
\end{proof}

\begin{theorem}\label{EinT}
A totally umbilical hypersurface in an Einstein space with
non-positive scalar curvature is biharmonic if and only if it is
minimal.
\end{theorem}
\begin{proof}
Take an orthonormal frame $\{e_1, \ldots, e_m,\xi\}$ of
$(N^{m+1},h)$ adapted to the hypersurface $M$ such that $A
e_i=\lambda_ie_i$, where $A$ is the Weingarten map of the
hypersurface and $\lambda_i$ is the principal curvature in the
direction $e_i$. Since $M$ is supposed to be totally umbilical,
i.e., all principal normal curvatures  at any point $p\in M$ are
equal to the same number $\lambda(p)$. It follows that
\begin{eqnarray}\notag
H=\frac{1}{m}\sum_{i=1}^m\langle A e_i,e_i\rangle=\lambda,\\\notag A
({\rm grad} H)=A(\sum_{i=1}^m(e_i \lambda)e_i)=\frac{1}{2}{\rm
grad}\, \lambda^2,\\\notag |A|^2=m \lambda^2.
\end{eqnarray}
The biharmonic hypersurface equation (\ref{Ein}) becomes
\begin{equation}\notag
\begin{cases}
\Delta \lambda-m\lambda^{3}+ \frac{r\lambda}{m+1}=0,\\
 (2+m)\,{\rm grad}\, \lambda^2
=0.
\end{cases}
\end{equation}
Solving the equation we have either $\lambda=0$ and hence $H=0$, or
$\lambda=\pm \sqrt{\frac{r}{m(m+1)}\;}$ is a constant and this
happens only if the scalar curvature is nonnegative, from which we
obtain the theorem.
\end{proof}
\begin{remark}
Our Theorem \ref{EinT} generalizes the results of \cite{BMO},
\cite{CMO2} and \cite{Di} about the totally umbilical biharmonic
hypersurfaces in a space form. It also implies that the generalized
B. Y. Chen's conjecture is true for totally umbilical hypersurfaces
in an Einstein space with non-positive scalar curvature. Note that
non-positive scalar curvature is a much weaker condition than
non-positive sectional curvature.
\end{remark}

\begin{corollary}\label{Rflat}
Any totally umbilical biharmonic hypersurface in a Ricci flat
manifold is minimal.
\end{corollary}
\begin{proof}
These follows from Theorem \ref{EinT} and the fact that a Ricci flat
manifold is an Einstein space with zero scalar curvature.
\end{proof}

\section{ Proper biharmonic foliations of codimension one}

In general, proper biharmonic maps as local solutions of a system of
$4$-th order PDEs are extremely difficult to unearth. Even in the
case of biharmonic submanifolds (viewed as biharmonic maps with
geometric constraints) few examples have been found. In this
section, we construct families of metrics that turns some foliations
of hypersurfaces into proper biharmonic foliations thus providing
infinitely many proper biharmonic hypersurfaces.

\begin{theorem}\label{Cflat}
For any constant $C$, let $N=\{(x_1,\ldots, x_m,z)\in
\r^{m+1}|z>-C\}$ denote the upper half space. Then, the conformally
flat space $(N,h=f^{-2}(z)(\sum_{i=1}^m{\rm d}{x_i}^{2}+{\rm
d}z^{2}))$ is foliated by proper biharmonic hyperplanes
$z=k\;\;(k\in \r, k>-C)$ if and only if $f(z)=\frac{D}{z+E}$, where
$E\geq C$ and $D\in \r\setminus\{0\}$.
\end{theorem}
\begin{proof}
Consider the isometric immersion $\varphi : (\mathbb{R}^m,
g)\longrightarrow (\mathbb{R}^{m+1},h=f^{-2}(z)(\sum_{i=1}^m{\rm
d}{x_i}^{2}+{\rm d}{z}^{2}))$ with $\varphi(x_1,\ldots,
x_m)=(x_1,\ldots, x_m,k)$ and $k$ being a constant, where the
induced metric $g$ with respect to the natural frame
$\partial_i=\frac{\partial}{\partial x_i},i=1,2, \ldots, m,
 \partial_{m+1}=\frac{\partial}{\partial z}$ has components
\begin{equation}\notag
g_{ij}=g(\partial_i,\partial_j)=h(d\varphi(\partial
_i),d\varphi(\partial _j))\circ\varphi=
\begin{cases}
f^{-2}(k),\;\;i=j,\\\notag 0\;\;\;i\ne j.
\end{cases}
\end{equation}
One can check that $ e_A=f(z)\partial_A\;\; (A=1,2,\ldots, m, m+1)$
constitute a local orthonormal frame  on $\mathbb{R}^{m+1}$ adapted
to the hypersurface $z=k$ with $\xi=e_{m+1}$ being the unit normal
vector field. A straightforward computation using Koszul's formula
gives the coefficients of the Levi-Civita connection of the ambient
space as

\begin{equation}\label{NCFE1}
(\nabla_{e_A}e_{B})=\left(\begin{array}{ccccc} f'e_{m+1}
&0&\ldots&0&-f'e_1
\\ 0 & f'e_{m+1}&\ldots&0&-f'e_2\\
\ldots&\ldots&\ldots&\ldots&\ldots
\\ 0 & 0&\ldots&f'e_{m+1}&-f'e_m\\
0&0&\ldots&0&0\\
\end{array}\right)_{(m+1)\times (m+1)}.
\end{equation}

Noting that $\xi=e_{m+1}$ is the unit normal vector field we can
easily compute the components of the second fundamental form as
\begin{eqnarray}\notag
h(e_i,e_j)=\langle\nabla_{e_i}e_j,e_{m+1}\rangle=\begin{cases}
f',\;\;i=j=1, 2,\ldots, m;\\
0,\;\; {\rm for\; all\; other\; cases}.
\end{cases},
\end{eqnarray}
from which we conclude that each of the hyperplane $z=k$ is a
totally umbilical hypersurface in the conformally flat space.\\

We compute the mean curvature of the hypersurface to have
\begin{equation}\notag
H=\frac{1}{m}\sum_{i=1}^mh(e_i,e_i)=f',
\end{equation}
 and the norm of the second fundamental form is given by
\begin{equation}\notag
|A|^2=\sum_{i=1}^m|h(e_i,e_i)|^2=mf'^2.
\end{equation}
Since $H$ depends only on $z$ we have ${\rm grad}_gH=
\sum_{i=1}^me_i(H)e_i=0$ and hence $\Delta_gH={\rm div} ({\rm
grad}_gH)=0$. Therefore, by Theorem \ref{MTH}, the biharmonic
equation of the isometric immersion reduces to the following system
\begin{equation}\label{NCONF}
\begin{cases}
-\,|A|^{2}+ {\rm
Ric}^N(\xi,\xi)=0,\\
  \,\sum_{i=1}^m\big({\rm Ric}^N\,(\xi,e_i)\big)e_i=0.
\end{cases}
\end{equation}

We can compute the Ricci curvature of the ambient space to have
\begin{eqnarray}\notag
&& {\rm Ric}\, (e_i,\xi)={\rm Ric}\,
(e_i,e_{m+1})=\sum_{j=1}^m\langle
R(e_{m+1},e_j)e_j,e_i\rangle=0,\;\;\forall \;i=1,2,\ldots,
m.\\\notag
 && {\rm Ric}\, (\xi,\xi)={\rm Ric}\,
(e_{m+1},e_{m+1})=\sum_{j=1}^m\langle
R(e_{m+1},e_j)e_j,e_{m+1}\rangle\\\notag&&=mff''-mf'^2.
\end{eqnarray}
Substitute these into Equation (\ref{NCONF}) we conclude that all
isometric immersions $\varphi : \mathbb{R}^m\longrightarrow
(\mathbb{R}^{m+1},h=f^{-2}(z)(\sum_{i=1}^m{\rm d}{x_i}^{2}+{\rm
d}{z}^{2}))$ with $\varphi(x_1,\ldots, x_m)=(x_1,\ldots, x_m,k)$ are
biharmonic if and only if
\begin{equation}\notag
ff''-2f'^2=0.
\end{equation}
This equation can be written as
\begin{equation}\notag
(\frac{f'}{f})'-(\frac{f'}{f})^2=0.
\end{equation}

 Solving this ordinary differential
equation we obtain the solutions $f(z)=\frac{D}{z+C}$ where $C, D$
are constants. Since the mean curvature of the hypersurface
$H=f'(k)$ is never zero we conclude that each of the hyperplanes
$z=k \;\;(k\ne -C)$ is a proper biharmonic hypersurface in the
conformally flat space $(N,h=(\frac{z+C}{D})^2(\sum_{i=1}^m{\rm
d}{x_i}^{2}+{\rm d}z^{2}))$. This completes the proof of the
theorem.
\end{proof}

\begin{theorem}\label{Pdouble}
The isometric immersion $\varphi : \mathbb{R}^2\longrightarrow
(\mathbb{R}^{3},h=e^{2p(z)}{\rm d}x^{2}+e^{2q(z)}{\rm d}y^{2}+{\rm
d}z^{2})$ with $\varphi(x,y)=(x,y,c)$ is  biharmonic if and only if
\begin{equation}\label{ODE}
p''+2p'^2+q''+2q'^2=0.
\end{equation}
In particular, for any positive constants $A, B, C, D$, the upper
half space $(\mathbb{R}^{3}_+=\{(x,y,z)|z>0\}$ with the metric
$h=(Az+B){\rm d}x^{2}+(Cz+D){\rm d}y^{2}+{\rm d}z^{2}$ is foliated
by proper biharmonic planes $z={\rm constant}$.
\end{theorem}
\begin{proof}
Consider the isometric immersion $\varphi :
\mathbb{R}^2\longrightarrow (\mathbb{R}^{3}_+,h=e^{2p(z)}{\rm
d}x^{2}+e^{2q(z)}{\rm d}y^{2}+{\rm d}z^{2})$ with
$\varphi(x,y)=(x,y,c)$ and $c>0$ being a constant. Using the
notations $\partial_1=\frac{\partial}{\partial x},
\partial_2=\frac{\partial}{\partial y}, \partial_3=\frac{\partial}{\partial z}$ we can easily check that the
induced metric is given by
\begin{equation}\notag
\begin{cases}
g_{11}=g(\partial_1,\partial_1)=h(d\varphi(\partial
_1),d\varphi(\partial _1))\circ\varphi=e^{2p(c)},\\\notag
g_{12}=g(\partial_1,\partial_2)=h(d\varphi(\partial_1),d\varphi(\partial
_2))\circ\varphi=0,\\\notag g_{22}=g(\partial_2,\partial
_2)=h(d\varphi(\partial _2),d\varphi(\partial
_2))\circ\varphi=e^{2q(c)}.
\end{cases}
\end{equation}
One can also check that $
e_1=e^{-p(z)}\partial_1,\;\;e_2=e^{-q(z)}\partial_2,\;\;e_3=\partial_3$
constitute an  orthonormal frame  on $\mathbb{R}^{3}_+$ adapted to
the surface $z=c$ with $\xi=e_3$ being the unit normal vector field.
A further computation gives the following Lie brackets
\begin{equation}\label{Li3}
[e_1,e_2]=0,\;\;[e_1,e_3]=p'e_1,\;\;[e_2,e_3]= q'e_2,
\end{equation}
and the coefficients of the Levi-Civita connection

\begin{equation}\label{CFE1}
\begin{array}{lll}
\nabla_{e_{1}}e_{1}=-p'e_3,\hskip0.7cm\nabla_{e_{1}}e_{2}=0, \hskip1cm\nabla_{e_{1}}e_{3}=p'e_1\\
\nabla_{e_{2}}e_{1}=0,\hskip1.5cm \nabla_{e_{2}}e_{2}=-q'e_3,\hskip0.8cm\nabla_{e_{2}}e_{3}=q'e_2\\
\nabla_{e_{3}}e_{1}=0,\hskip1.5cm\nabla_{e_{3}}e_{2}=0,\hskip1.2cm\nabla_{e_{3}}e_{3}=0.\\
\end{array}
\end{equation}

Noting that $\xi=e_3$ is the unit normal vector field we can compute
the components of the second fundamental form as
\begin{eqnarray}\notag
h(e_1,e_1)=\langle\nabla_{e_1}e_1,e_3\rangle=-p',
\\\notag h(e_1,e_2)= \langle\nabla_{e_1}e_2,e_3\rangle=0,\\\notag
h(e_2,e_2)= \langle\nabla_{e_2}e_2,e_3\rangle=-q'.
\end{eqnarray}

From these we obtain the mean curvature of the isometric immersion
\begin{equation}\label{Mean}
H=\frac{1}{2}(h(e_1,e_1)+h(e_2,e_2))=-(p'+q')/2,
\end{equation}
 and
the norm of the second fundamental form
\begin{equation}\notag
|A|^2=\sum_{i=1}^2|h(e_i,e_i)|^2=p'^2+q'^2.
\end{equation}
Since $H$ depends only on $z$ we have ${\rm grad}_gH=
e_1(H)e_1+e_2(H)e_2=0$ and hence $\Delta_gH={\rm div} ({\rm
grad}_gH)=0$. Therefore, by Theorem \ref{MTH}, the biharmonic
equation of the isometric immersion reduces to Equation
(\ref{NCONF}) with $m=2$. To compute the Ricci curvature of the
ambient space we can use (\ref{Li3}) and (\ref{CFE1}) to have
\begin{eqnarray}\notag
&& {\rm Ric}\, (e_1,\xi)={\rm Ric}\, (e_1,e_3)=\langle
R(e_3,e_2)e_2,e_1\rangle=0,\\\notag &&{\rm Ric}\, (e_2,\xi)={\rm
Ric}\,(e_2,e_3)=\langle R(e_3,e_1)e_1,e_3\rangle=0,\\\notag && {\rm
Ric}\, (\xi,\xi)={\rm Ric}\, (e_3,e_3)=\langle
R(e_3,e_1)e_1,e_3\rangle+\langle R(e_3,e_2)e_2,e_3\rangle\\\notag
&&=-p''-p'^2-q''-q'^2.
\end{eqnarray}
Substitute these into Equation (\ref{NCONF}) with $m=2$ we conclude
that the isometric immersion $\varphi : \mathbb{R}^2\longrightarrow
(\mathbb{R}^{3},h=e^{2p(z)}{\rm d}x^{2}+e^{2q(z)}{\rm d}y^{2}+{\rm
d}z^{2})$ with $\varphi(x,y)=(x,y,c)$ is  biharmonic if and only if
Equation (\ref{ODE}) holds, which gives the first statement of the
Theorem. The second statement of the theorem is obtained by looking
for the solutions of (\ref{ODE}) satisfying $p''+2p'^2=0$ and
$q''+2q'^2=0$. In fact, we have special solutions
$p(z)=\frac{1}{2}\ln (Az+B)$ and $q(z)=\frac{1}{2}\ln (Cz+D)$ with
positive constants $A, B, C, D$. By (\ref{Mean}) and the choice of
positive constants $A, B, C, D$ we see that the mean curvature of
the surface $z=c$ is $ H=-\frac{2ACz+AD+BC}{2(Az+B)(Cz+D)}\ne 0$ and
hence each such surface is a non-minimal biharmonic surface. This
completes the proof of the theorem.
\end{proof}
\begin{remark}
One can check that Theorem \ref{Pdouble} has a generalization to a
higher dimensional space $\mathbb{R}^{m}_+$ for $m>3$.
\end{remark}
\begin{example}\label{Ex1}
Let $\lambda(t)=\sqrt {A t+B}$ with positive constants $A,B$. Then,
the warped product space $N=( S^2\times \r^+, h=
\lambda^2(t)g^{S^2}+{\rm d}t^2)$ is foliated by the spheres
$(S^2\times \{t\}, \lambda^2(t)g^{S^2})$ each of which is a totally
umbilical proper biharmonic surface.
\end{example}

In fact, to see what is claimed in the Example \ref{Ex1}, we
parametrize the unit sphere $S^{2}$ by spherical polar coordinates:
\begin{equation}\notag
\mathbb{R}\times\mathbb{R}\ni(\rho,\theta)\longrightarrow (\cos
\rho, \sin \rho\, e^{i \theta})\in \mathbb{R}^{3}.
\end{equation}
Then, the standard metric can be written as $g^{S^2}={\rm d}\rho^2+
\sin^2\rho\; {\rm d}\theta^2$ and hence the warped product metric on
$N$ takes the form $h=\lambda^2(t){\rm d}\rho^2+
\lambda^2(t)\sin^2\rho\; {\rm d}\theta^2+{\rm d}t^2$. Consider the
isometric immersion $\varphi : S^2\longrightarrow
(\mathbb{R}^+\times S^2,{\rm d}t^2+\lambda^2(t)g^{S^2})$ with
$\varphi(\rho,\theta)=(\rho,\theta, c)$ and $c>0$ being a constant.
Using the notations $
\partial_1=\frac{\partial}{\partial \rho}, \partial_2=\frac{\partial}{\partial \theta},\;
\partial_3=\frac{\partial}{\partial t}$ we can easily check that the
induced metric is given by
\begin{equation}
\begin{cases}
g_{11}=g(\partial_1,\partial_1)=h(d\varphi(\partial
_1),d\varphi(\partial _1))\circ\varphi=\lambda^2(c),\\\notag
g_{12}=g(\partial_1,\partial_2)=h(d\varphi(\partial_1),d\varphi(\partial
_2))\circ\varphi=0,\\\notag g_{22}=g(\partial_2,\partial
_2)=h(d\varphi(\partial _2),d\varphi(\partial
_2))\circ\varphi=\lambda^2(c)\sin^2\rho.
\end{cases}
\end{equation}
Using the orthonormal frame $
e_1=\lambda^{-1}(t)\partial_1,\;\;e_2=(\lambda(t)\sin
\rho)^{-1}\,\partial_2,\;\;e_3=\partial_3$ we have the Lie brackets
\begin{equation}\notag
[e_1,e_2]=-\frac{\cot \rho}{\lambda}
e_2,\;\;[e_1,e_3]=fe_1,\;\;[e_2,e_3]=fe_2,\;\;,
\end{equation}
 where and in the sequel we use the notation $f=(\ln \lambda)'=\frac{\lambda'}{\lambda}$.
 Clearly, $e_1, e_2, \xi=e_3=\partial_3$
 constitute a local orthonormal frame of $N$ adapted to the surface with $\xi=e_3=\partial_3$ being the unit
normal vector field of the surface. We can use the Kozsul formula to
compute the components of the second fundamental form as
\begin{eqnarray}\notag
h(e_1,e_1)&=&
\langle\nabla_{e_1}e_1,\xi\rangle=\langle\nabla_{e_1}e_1,e_3\rangle\\\notag
&=&\frac{1}{2}\left(-\langle e_1,[e_1,e_3]\rangle-\langle
e_1,[e_1,e_3]\rangle+\langle e_3,[e_1,e_1]\rangle\right)=-f,\\\notag
h(e_1,e_2)&=&
\langle\nabla_{e_1}e_2,\xi\rangle=\langle\nabla_{e_1}e_2,e_3\rangle=0,\\\notag
h(e_2,e_2)&=&
\langle\nabla_{e_2}e_2,\xi\rangle=\langle\nabla_{e_2}e_2,e_3\rangle=-f,
\end{eqnarray}
from which we conclude that each of such spheres is totally umbilical surface in $N$.\\

Notice that the mean curvature of the isometric immersion is
$H=\frac{1}{2}(h(e_1,e_1)+h(e_2,e_2))=-f$, and the norm of the
second fundamental form $|A|^2=\sum_{i=1}^2|h(e_i,e_i)|^2=2f^2$,
which depend only on $t$. It follows that ${\rm grad}_gH=0$ and
$\Delta_gH=0$. Therefore, by Theorem \ref{MTH}, the proper
biharmonic equation of the isometric immersion reduces to Equation
(\ref{NCONF}) with $m=2$.\\

On the other hand, using the Ricci curvature formula (see e.g.,
\cite{Be}) of the warped product $M=B\times_{\lambda}F$ we have
\begin{eqnarray}\notag
&& {\rm Ric}\, (e_1,\xi)={\rm Ric}\, (e_1,e_3)=0,\;\;{\rm Ric}\,
(e_2,\xi)={\rm Ric}\,(e_2,e_3)=0,\\\notag && {\rm Ric}\,
(\xi,\xi)={\rm Ric}\, (e_3,e_3)={\rm
Ric}^{\mathbb{R}}(e_3,e_3)-\frac{2}{\lambda}{\rm
Hess}_{\lambda}(e_3,e_3)\\\notag &&
=-\frac{2}{\lambda}\left(e_3(e_3\lambda)-{\rm
d}\lambda(\nabla_{e_3}e_3)\right)=-\frac{2\lambda''}{\lambda}.
\end{eqnarray}
Substitute these into Equation (\ref{NCONF}) with $m=2$ we conclude
that the isometric immersion $\varphi : S^2\longrightarrow
(S^2\times \mathbb{R}^+,\lambda^2(t)g^{S^2}+{\rm d}t^2)$ with
$\varphi(\rho,\theta)=(\rho,\theta, c)$ is  biharmonic if and only
if

\begin{equation}\notag
-2\left(\frac{\lambda'}{\lambda}\right)^2-\frac{2\lambda''}{\lambda}=0.
\end{equation}
Solving this final equation we have $\lambda(t)=\sqrt{At+B\,}$ and
from which we obtain the proposition.
\begin{remark}
The author would like to thank the referee for informing him that
the biharmonicity of the inclusion maps in Example \ref{Ex1} can be
obtained as a particular case of Corollary 3.4 in \cite{BMO0} which
was proved by a different method.
\end{remark}

\section{Biharmonic cylinders of a Riemannian submersion}
Let $\pi: (M^3, g)\longrightarrow (N^2, h)$ be a Riemannian
submersion with totally geodesic fibers from a complete manifold.
Let $\alpha: I\longrightarrow (N^2, h)$ be an immersed regular curve
parametrized by arclength. Then $\Sigma =\bigcup_{t\in
I}\pi^{-1}(\alpha(t))$ is a surface in $M$ which can be viewed as a
disjoint union of all horizontal lifts of the curve $\alpha$. Let
$\{{\bar X}=\alpha', {\bar \xi}\}$ be a Frenet frame along $\alpha$
and ${\bar \kappa}$ be the geodesic curvature of the curve. Then,
the Frenet formula for $\alpha$ is give by
\begin{equation}\notag
\begin{cases}
{\tilde\nabla}_{\bar X} {\bar X}={\bar\kappa} {\bar \xi},\\
{\tilde \nabla}_{\bar X} {\bar \xi}=-{\bar\kappa} {\bar X},
\end{cases}
\end{equation}
where  ${\tilde\nabla}$ denote the Levi-Civita connection of $(N,
h)$. Let $\beta: I\longrightarrow (M^3, g)$ be a horizontal lift of
$\alpha$. Let $X$ and $\xi$ be the horizontal lifts of ${\bar X}$
and ${\bar \xi}$ respectively. Let $V$ be the unit vector field
tangent to the fibers of the submersion $\pi$. Then $\{X, \xi, V\}$
form an orthonormal frame of $M$ adapted to the surface with $\xi$
being the unit normal vector of the surface. Notice that the
restriction of this frame to the curve $\beta$ is the Frenet frame
along $\beta$. Therefore, we have the Frenet formula along $\beta$
given by
\begin{equation}\label{FR}
\begin{cases}
\nabla_X X=\kappa \xi,\\
\nabla_X \xi=-\kappa X+\tau V,\\
\nabla_X V=-\tau \xi,\\
\end{cases}
\end{equation}
where $\nabla$ denotes the Levi-Civita connection of $(M,g)$. Since
a Riemannian submersion preserves the inner product of horizontal
vector fields we can check that $ \kappa={\bar\kappa}\circ \pi$ and
$\tau=\langle \nabla_{X}\xi ,V\rangle=\langle A_{X}\xi, V\rangle$
(where, $A$ is the A-tensor of the Riemannian submersion, c.f.
\cite{On}) is the torsion of the horizontal lift which vanishes if
the Riemannian submersion has integrable horizontal distribution. In
what follows we are going to use the orthonormal frame $\{X, \xi,
V\}$ to compute the mean curvature, second fundamental form, and
other terms
that appear in the biharmonic equation of the surface $\Sigma$.\\
Using (\ref{FR}) we have
\begin{eqnarray}\notag
&& A(X)=-\langle \nabla_X \xi, X\rangle X-\langle \nabla_X \xi,
V\rangle V=\kappa X -\tau V,\\\notag && A(V)=-\langle \nabla_V \xi,
X\rangle X-\langle \nabla_V \xi, V\rangle V=-\tau X ;\\\notag &&
b(X,X)=\langle A(X), X\rangle=\kappa,\;\;b(X,V)=\langle A(X),
V\rangle=-\tau,\\\notag && b(V,X)=\langle A(V),
X\rangle=-\tau,\;\;b(V,V)=\langle A(V), V\rangle=0;\\\notag &&
H=\frac{1}{2}(b(X,X)+b(V,V))=\frac{\kappa}{2},\\\notag && A({\rm
grad}\,H)=A(X(\frac{\kappa}{2})X+V(\frac{\kappa}{2})V)=X(\frac{\kappa}{2})A(X)=\frac{\kappa'}{2}(\kappa
X -\tau V);\\\notag && \Delta H=XX(H)-(\nabla_XX)
H+VV(H)-(\nabla_VV) H=\frac{\kappa''}{2};\\\notag &&
|A|^2=(b(X,X))^2+(b(X,V))^2+( b(V,X))^2+(b(V,V))^2=\kappa^2+2\tau^2.
\end{eqnarray}
Substituting these into the biharmonic hypersurface Equation
(\ref{BHEq}) we conclude that the surface $\Sigma$ is biharmonic in
$(M^3, g)$ if and only if
\begin{equation}\notag
\begin{cases}
\frac{\kappa''}{2}-\frac{\kappa}{2}(\kappa^2+2\tau^2)+\frac{\kappa}{2}{\rm
Ric}^M(\xi,\xi)=0,\\ \kappa'(\kappa X -\tau
V)+\frac{\kappa\kappa'}{2}X-\kappa\,{\rm Ric}^M(\xi,X)X-\kappa{\rm
Ric}^M(\xi,V)V=0,
\end{cases}
\end{equation}
which are equivalent to
\begin{equation}\label{Last}
\begin{cases}
\kappa''-\kappa(\kappa^2+2\tau^2)+\kappa{\rm Ric}^M(\xi,\xi)=0,\\
3\kappa'\kappa -2\kappa\,{\rm Ric}^M(\xi,X)=0,\\ \kappa'\tau
+\kappa{\rm Ric}^M(\xi,V)=0.
\end{cases}
\end{equation}

Applying Equation (\ref{Last}) to Hopf fiberation $\pi:
S^3\longrightarrow S^2$ we have the following corollary which
recovers Proposition 3.1 in \cite{In}.
\begin{corollary}
There is no proper biharmonic Hopf cylinder in $S^3$.
\end{corollary}

Lastly, applying Equation (\ref{Last}) to submersions $\pi:S^2\times
\r\longrightarrow S^2$ and  $\pi:H^2\times \r\longrightarrow H^2$ we
can have
\begin{corollary}
(1) The Hopf cylinder $\Sigma =\bigcup_{t\in I}\pi^{-1}(\alpha(t))$
is a proper biharmonic surface in $S^2\times \r$ if and only if the
directrix $\alpha: I\longrightarrow (S^2, h)$ is (a part of) a
circle in $S^2$ with radius $\sqrt{2}/2$; (2) The Hopf cylinder
$\Sigma =\bigcup_{t\in I}\pi^{-1}(\alpha(t))$ is biharmonic in
$H^2\times \r$ if and only if it is minimal.
\end{corollary}

\end{document}